\documentclass[12pt]{article}

\begin{document}

\begin{titlepage}

 \vspace*{-2cm}

\vspace{.5cm}

\begin{centering}

\huge{On the tensor structure of BRST differential and it's
application}

\vspace{.5cm}

\large  {Jining Gao }\\

\vspace{.5cm}

Department of Mathematics, Shanghai Jiaotong University, Shanghai
,P. R. China

\vspace{.5cm}

\begin{abstract}
In this paper,we compute tensor structure of BRST differential and
use this tensor representation we give out $CL_{\infty}$ algebra
differential and  $GA_{\infty}$ differential which are
generalization of Chevalley-Eilenberg differential and Hochchild
differential respectively.

\end{abstract}

\end{centering}

\end{titlepage}

\pagebreak

\def\lh{\hbox to 15pt{\vbox{\vskip 6pt\hrule width 6.5pt height 1pt}
  \kern -4.0pt\vrule height 8pt width 1pt\hfil}}
\def\blob{\mbox{$\;\Box$}}
\def\qed{\hbox{${\vcenter{\vbox{\hrule height 0.4pt\hbox{\vrule width
0.4pt height 6pt \kern5pt\vrule width 0.4pt}\hrule height
0.4pt}}}$}}

\newtheorem{theorem}{Theorem}
\newtheorem{lemma}[theorem]{Lemma}
\newtheorem{definition}[theorem]{Definition}
\newtheorem{corollary}[theorem]{Corollary}
\newtheorem{proposition}[theorem]{Proposition}
\newcommand{\proof}{\bf Proof.\rm}

\section{Introduction}
In theoretical physics BRST differential\cite{HT} (or BRST symmetry)
is usually represented by ghost variables and when ghost variables
correspond to exterior differential forms in differential geometry,
BRST differential corresponds to an exterior differential in some
specific case. Although exterior differential forms is main language
in modern differential geometry, tensor representation still plays
an important role in geometry and physics as an another equivalent
language. The correspondence between ghost (or exterior differential
forms) representations and tensor representations is very
transparent when we consider a special case of BRST differential-
 Chevalley-Eilenberg differential as follows:

  Chevalley-Eilenberg differential or (Lie algebra
differential) was introduced and developed in the fifties and
sixties of the last century\cite{CE},\cite{AL}. Let's first recall
it's definition.

Assume that $L$ is a Lie algebra, $V$ is a representation space with
representation $\rho.$  Set $C^{n}(L,V)={\omega _{n}: L\times
L\cdots \times L\rightarrow V}$. The  Lie algebra differential
$S_{n} :C^{n}(L,V)\rightarrow C^{n+1}(L,V)$ is defined by its action
on the cochain:
$$
(S\omega_{n})(X_1,\cdots
X_{n+1})=\sum_{i=1}^{n+1}(-1)^{i+1}\rho(X_{i})(\omega(X_1,\cdots\hat{X}_{i},\cdots
X_{n+1})$$
\begin{eqnarray}
+\sum_{j,k=1}^{n+1}(-1)^{j+k}\omega([X_j,X_k],X_{1}\cdots\hat{X}_{j},\cdots
X_{k}\cdots X_{n+1})
\end{eqnarray}

Where $X_1,\cdots X_{n+1}\in L$ and the Lie algebra cohomology can
be induced by the complex $(C^{n},s_{n})$. Since complex
$C^{n}(L,V)$ consists of multi-linear maps which are represented by
various tensors,we call above description of Chevalley-Eilenberg
differential tensor representation. If $L$ is a finite dimensional
Lie algebra, the Chevalley-Eilenberg differential defined above can
have an equivalent but more neat description so called ghost
representation or BRST approach as follows: Let ${X_{i}}$ be a basis
of $L$, for every $X_{i}$, we introduce a ghost variable $\eta^{i}$
such that
$$\eta^{i}\eta^{j}=-\eta^{j}\eta^{i}$$ and

\begin{eqnarray}
\tilde{S}
=\eta^{i}\rho(X_{i})+\frac{1}{2}c_{ij}^{k}c^{j}c^{i}\frac{\partial}{\partial
c^k}
\end{eqnarray}
Although both definitions of Chevalley-Eilenberg differential looks
different, the nilpotency is only determined by following two basic
facts:

(a)$\rho$ is a representation of Lie algebra $L$ with representation
space $V$ i.e. for any $X_{i},X_{j}\in L$
$$\rho([X_{i},X_{j}])=[\rho(X_{i}),\rho(X_{j})]$$

(b)$L$ is a Lie algebra

To author's knowledge,BRST differential just has ghost
representation so far. We wonder if there exists tensor
representations just like in the case of CE differential,or
equivalently,if there exists similar conditions as (a) and (b) to
determine the nilpotency of BRST differential.
 The purpose of this
paper is to give a complete answer to above questions . First of
all, by a lengthy computation ,we obtain nilpotent equations of
expansion coefficients for general BRST differential, secondly , we
assign a multi-linear map to every  expansion coefficient and
nilpotent equations transform to equations in term of those
multi-linear maps. Those equations are obviously divided into two
groups as (a) and (b)but much more complicated. We will find
equations group (b) are exactly structure equations of commutator sh
lie structure which is called $CL_{\infty}$ algebra or in more
general,$GA_{\infty}$ algebra in this paper, and according equations
group (a), we introduce linear representations of $CL_{\infty}$
algebra and $GA_{\infty}$ algebra . With all above preliminary , we
can define so called $CL_{\infty}$ algebra differential and
$GA_{\infty}$ algebra differential. Just like Lie algebra
 cohomology is a power tool in Lie algebra representation,
 we may expect  $GA_{\infty}$ algebra
 cohomology has similar
 application in representation of $A_{\infty}$ algebra in future.

\section{Symmetry and skew-symmetry operators}
In this section, we will give some preliminary about symmetry and
skew-symmetry tensor which is usually encountered in the computation
of tensor derivative. First of all, it's necessary to introduce
skew-symmetry or symmetry operator. For any given tensor coefficient
$f_{a_{1}\cdots a_{n}}$ Let$$\hat{S}(f_{a_{1}\cdots a_{n}})=
\sum_{\sigma\in s_n}(-1)^{\sigma}e(\sigma)(f_{\sigma(a_{1})\cdots
\sigma(a_{n})})$$ and
$$\tilde{S}(f_{a_{1}\cdots a_{n}})=
\sum_{\sigma\in s_n}e(\sigma)(f_{\sigma(a_{1})\cdots
\sigma(a_{n})})$$ we have that
\begin{lemma}
$\hat{S}$ and
$\tilde{S}$ are skew symmetry and symmetry operators respectively.
\end{lemma}
\begin{proof}
Let's first prove that $\hat{s}$ is a skew symmetry operator
\begin{eqnarray}
\hat{S}(f_{\tau(a_{1})\cdots \tau(a_{n})}=\sum_{\sigma\in
s_n}(-1)^{\sigma}e(\sigma)(f_{\sigma(\tau(a_{1}))\cdots
\sigma(\tau(a_{n}))})\nonumber\\
=\sum_{\sigma\in
s_n}(-1)^{\sigma}e((\sigma\tau)\tau^{-1})(f_{(\sigma\tau)(a_{1})\cdots
(\sigma\tau)(a_{n})})\nonumber\\
=\sum_{\sigma\in
s_n}(-1)^{\sigma}e((\sigma\tau))e(\tau^{-1})(f_{(\sigma\tau)(a_{1})\cdots
(\sigma\tau)(a_{n})})\nonumber\\
=\sum_{\sigma\in
s_n}(-1)^{\sigma}(-1)^{\sigma\tau}(-1)^{\sigma\tau}e((\sigma\tau))e(\tau^{-1})(f_{(\sigma\tau)(a_{1})\cdots
(\sigma\tau)(a_{n})})\nonumber\\
=\sum_{\sigma\in s_n}(-1)^{\tau}
(-1)^{\sigma\tau}e((\sigma\tau))e(\tau^{-1})(f_{(\sigma\tau)(a_{1})\cdots
(\sigma\tau)(a_{n})})\nonumber\\
=(-1)^{\tau}e(\tau^{-1})[\sum_{\sigma\in s_n}
(-1)^{\sigma\tau}e((\sigma\tau))e(\tau^{-1})(f_{(\sigma\tau)(a_{1})\cdots
(\sigma\tau)(a_{n})})]
\end{eqnarray}
Let $\tau^{'}=\sigma\tau$, when $\sigma$ runs through all elements
of $s_n$, so does $\tau^{'}$ thus
$$
\hat{S}(f_{\tau(a_{1})\cdots
\tau(a_{n}})=(-1)^{\tau}e(\tau^{-1})[\sum_{\tau^{'}\in s_n}
(-1)^{\sigma\tau^{'}}e((\sigma\tau^{'}))e(\tau^{-1})(f_{(\sigma\tau^{'})(a_{1})\cdots
(\sigma\tau^{'})(a_{n}))}]$$
\begin{eqnarray}
 =(-1)^{\tau}e(\tau^{-1})\hat{S}(f_{a_{1}\cdots a_{n}})
\end{eqnarray}

Similarly,we can prove$\check{S}$ is a symmetry operator
\end{proof}
\section{Structure of BRST differential and it's tensor representation }

In this section, we will give out structure equations of BRST
differential and express it via muti-linear maps. Assume that $V$ is
a graded space and $V_{i}(i \in I)$ is a basis of $V$., the
underlying field is $F$. Suppose there are series of skew
multi-linear maps:
\begin{eqnarray}
l_n : \bigotimes^{n }V\rightarrow V
\end{eqnarray}
i.e
\begin{eqnarray}
l_{n}(v_{\sigma_{(1)}}\otimes v_{\sigma_{(2)}}\cdots \otimes
v_{\sigma_{(n)}}) =(-1)^{\sigma}e(\sigma)l_{n}(v_{(1)},
v_{(2)}\cdots v_{(n)})
\end{eqnarray}
Assume that  \begin{eqnarray} l_{n}(v_{i_{1}}, v_{i_{2}} \cdots
v_{i_{n}})=\sum C_{i_{1}\cdots i_{n}}^{j}v_{j}\label{a}
\end{eqnarray}
 where $C_{i_{1}\cdots
i_{n}}^{j}\in F$. Consider dual space $V^*=span(\eta^i,i\in I)$ with
$deg \eta^{i}=deg v_{i}+1$. Let $R=F[\eta^i,i\in I]$, there is a
natural ring structure on $R$ with following commutative law:
\begin{eqnarray}
\eta^a\eta^b=-(-1)^{(deg a)(deg b)}\eta^b\eta^a \end{eqnarray} From
now on,we will study an odd derivative $S$ determined by the
following structure equations.
\begin{eqnarray}
S\eta^j=-\sum_{k=1}^{\infty}\frac{1}{k!}C_{i_{1}\cdots
i_{n}}^{j}\eta^{i_1}\cdots\eta^{i_k}
\end{eqnarray}
Where the skew symmetry coefficients $C_{i_{1}\cdots i_{n}}^{j}$ are
given by \ref{a}. Extend operation of $S$ on above generators to all
elements of $R$ by the following Lebnitz rlue of odd derivative.
\begin{eqnarray}
S(\eta^j\eta^k)=S(\eta^j)\eta^{k}-(-1)^{deg\eta^j}\eta^{j} S(\eta^k)
\end{eqnarray}
(or in short set $deg\eta^j=\eta^j$) \newline We will determine when
$S$ becomes a differential. Before giving out the condition, we need
following lemma.
\begin{lemma}

\begin{eqnarray}
S(\eta^{i_1}\cdots\eta^{i_k})=\sum_{l=1}^{k}(-1)^{l-1}(-1)^{\eta^{i_1}+\cdots+\eta^{i_{l-1}}}
\eta^{i_1}\cdots S(\eta^{i_l})\cdots \eta^{i_k}
\end{eqnarray}
\end{lemma}
\begin{proof}We use induction method to derive it. It's obvious for
$k=2$. Suppose that the formula holds for $k-1,(k\geq 3)$ then
\begin{eqnarray}
& & S(\eta^{i_1}\cdots\eta^{i_k})=S[(\eta^{i_1}\cdots\eta^{i_{k-1}})\eta^{i_k}]\nonumber\\
& &
=S(\eta^{i_1}\cdots\eta^{i_{k-1}})\eta^{i_k}-(-1)^{deg(\eta^{i_1}
\cdots\eta^{i_{k-1}})}\eta^{i_1}\cdots\eta^{i_{k-1}}S\eta^{i_k}\nonumber\\\end{eqnarray}

$$
=[\sum_{l=1}^{k-1}(-1)^{l-1}(-1)^{\eta^{i_1}+\cdots+\eta^{i_{l-1}}}
\eta^{i_1}\cdots S(\eta^{i_l})\cdots \eta^{i_{k-1}}]\eta^{i_k}
$$

\begin{eqnarray}
-(-1)^{\eta^{i_1}+\cdots+\eta^{i_{k-1}}+k-2} \eta^{i_1}\cdots
S(\eta^{i_l})\cdots
\eta^{i_{k-1}}]S\eta^{i_k}\nonumber\\
\end{eqnarray}

$$
=\sum_{l=1}^{k-1}(-1)^{l-1}(-1)^{\eta^{i_1}+\cdots+\eta^{i_{l-1}}}
\eta^{i_1}\cdots S(\eta^{i_l})\cdots \eta^{i_k}$$
\begin{eqnarray}
+(-1)^{k-1}(-1)^{\eta^{i_1} +\cdots+\eta^{i_{k-1}}} \eta^{i_1}\cdots
S(\eta^{i_l})\cdots \eta^{i_{k-1}}]S\eta^{i_k}\nonumber\\
\end{eqnarray}
\end{proof}
${\bf Remark :}$ If we consider another type of derivative, i.e.
$$ S(\alpha\beta)=S(\alpha)\beta+(-1)^{deg\alpha}S(\beta)$$
the Lemma 2 should be modified to
\begin{eqnarray}
S(\eta^{i_1}\cdots\eta^{i_k})=\sum_{l=1}^{k}(-1)^{\eta^{i_1}+\cdots+\eta^{i_{l-1}}}
\eta^{i_1}\cdots S(\eta^{i_l})\cdots \eta^{i_k}
\end{eqnarray}
and $$ deg(\eta^{i}\eta^{j})=deg\eta^{i}+deg\eta^{j}$$. With above
preliminary, we have
\begin{theorem} $S^2\eta^{j}=0$ iff
\begin{eqnarray}
\sum_{\sigma\in S_{k+m-1},1\leq l\leq
k}(-1)^{A(\sigma)}e(\sigma)C_{a,i_{\sigma(1)}\cdots
i_{\sigma(l+m)}\cdots i_{\sigma(k+m-1)}}^{j}C_{i_{\sigma(l)}\cdots
i_{\sigma(l+m-1}}^{a}=0
\end{eqnarray}
Where $$A(\sigma)=\sigma+i_{\sigma(l)}+\cdots
+i_{\sigma(l+m-1}+(a,i_{\sigma(l)},\cdots i_{\sigma(l-1})$$
\end{theorem}
\begin{proof}
\begin{eqnarray}
S^2\eta^{j}=S(-\sum_{k=1}^{\infty}\frac{1}{k!}C_{i_{1}\cdots
i_{k}}^{j}\eta^{i_1}\cdots\eta^{i_k})\nonumber\\
=-\sum_{k=1}^{\infty}\frac{1}{k!}C_{i_{1}\cdots
i_{n}}^{j}S(\eta^{i_1}\cdots\eta^{i_k})
\end{eqnarray}
By the Lemma 2:
\begin{eqnarray}
& &
S(\eta^{i_1}\cdots\eta^{i_k})=\sum_{l=1}^{k}(-1)^{\eta^{i_1}+\cdots+\eta^{i_{l-1}}}
\eta^{i_1}\cdots S(\eta^{i_l})\cdots \eta^{i_k}\nonumber\\
& & =\sum_{l=1}^{k}(-1)^{l-1}(-1)^{\eta^{i_1}+\cdots+\eta^{i_{l-1}}}
\eta^{i_1}\cdots [-\sum_{m=1}^{\infty}\frac{1}{m!}C_{j_{1}\cdots
j_{m}}^{i_l}\eta^{j_1}\cdots\eta^{j_m} ]\cdots \eta^{i_k}\nonumber\\
& &
=\sum_{l=1}^{k}[-\sum_{m=1}^{\infty}\frac{1}{m!}(-1)^{l-1}(-1)^{\eta^{i_1}+\cdots+\eta^{i_{l-1}}}
C_{j_{1}\cdots j_{m}}^{i_l}\eta^{i_1}\cdots
\eta^{i_{l-1}}[\eta^{j_1}\cdots\eta^{j_m}]\eta^{i_{l+1}}\cdots
\eta^{i_k}]\nonumber\\
\end{eqnarray}
Thus
$$
S^2\eta^{j}=\sum_{m=1}^{\infty}\sum_{k=1}^{\infty}\frac{1}{m!k!}\sum_{l=1}^{k}(-1)^{l-1}(-1)^{\eta^{i_1}+\cdots+\eta^{i_{l-1}}}
C_{i_{1}\cdots i_{k}}^{j}C_{j_{1}\cdots j_{m}}^{i_l}\eta^{i_1}\cdots
\eta^{i_{l-1}}[\eta^{j_1}\cdots\eta^{j_m}]\eta^{i_{l+1}}\cdots
\eta^{i_k}
$$
In order to simplify $S^2\eta^{j}$ further,let's first introduce
some notations. We denote permutation parity from $(i_1,\cdots
i_{l-1},i_l \cdots i_n)$ to $(i_l,i_1,\cdots i_{l-1},i_{l+1} \cdots
i_n)$$ by $$(i_l,i_1,\cdots i_{l-1})$ and denote permutation sign by
$e(i_l,i_1,\cdots i_{l-1})$and
\begin{eqnarray}
S^2\eta^{j}=\sum_{m=1}^{\infty}\sum_{k=1}^{\infty}\frac{1}{m!k!}\sum_{l=1}^{k}A(i_l,i_1,\cdots
i_{l-1}) C_{i_{l},i_{1}\cdots \hat {i_{l}}\cdots
i_{k}}^{j}C_{j_{1}\cdots j_{m}}^{i_l}\eta^{i_1}\cdots
\eta^{i_{l-1}}[\eta^{j_1}\cdots\eta^{j_m}]\eta^{i_{l+1}}\cdots
\eta^{i_k}\nonumber\\
=\sum_{m=1}^{\infty}\sum_{k=1}^{\infty}\frac{1}{m!k!}\sum_{l=1}^{k}
A(a,i_1,\cdots i_{l-1})C_{a,i_{1},\cdots i_{l-1}, i_{l+m}\cdots
i_{k+m-1}}^{j}C_{i_{l}\cdots i_{l+m-1}}^{a}\eta^{i_1}\cdots
\eta^{i_{k+m-1}}
\end{eqnarray}
Where $$A(i_l,i_1,\cdots
i_{l-1})=(-1)^{l-1}(-1)^{\eta^{i_1}+\cdots+\eta^{i_{l-1}}}
(-1)^{(i_l,i_1,\cdots i_{l-1})}e(i_l,i_1,\cdots i_{l-1})$$ Let

$$f_{i_{1},\cdots
i_{k+m-1}}^{j}= A(a,i_1,\cdots i_{l-1})C_{a,i_{1},\cdots i_{l-1},
i_{l+m}\cdots i_{k+m-1}}^{j}C_{i_{l}\cdots i_{l+m-1}}^{a}$$ and
anti-symmetrize $$f_{i_{1},\cdots i_{k+m-1}}^{j}$$ by the operator
$\hat{S}$
\begin{eqnarray}
\hat{S}(f_{i_{1},\cdots i_{k+m-1}}^{j})=\sum_{\sigma \in
S_{k+m-1}}(-1)^{\sigma}e(\sigma)f_{i_{\sigma(1)},\cdots
i_{\sigma(k+m-1)}}^{j}
\end{eqnarray}
Thus we have

\begin{eqnarray}
S^{2}\eta^{j}=\sum_{m=1}^{\infty}\sum_{k=1}^{\infty}\frac{1}{m!k!(m+k-1)!}
\sum_{\sigma\in S_{k+m-1},1\leq l \leq
k}(-1)^{\sigma}e(\sigma)f_{i_{\sigma(1)},\cdots
i_{\sigma(k+m-1)}}^{j} \eta^{i_1}\cdots \eta^{i_{k+m-1}}
\end{eqnarray}
Then follows the theorem.
\end{proof}
Because the structure constants come from multi-linear maps $l_n$,
we have:
\begin{proposition}
$S^2 =0$ iff following equations hold:
\begin{eqnarray}
\frac{1}{m!k!} \sum_{\sigma\in S_{k+m-1},1\leq l\leq k}
l_{k}[v_{\sigma(1)}\cdots v_{\sigma(l-1)},l_{m}[v_{\sigma(l)}\cdots
v_{\sigma(l+m-1)}],v_{\sigma(l+m)}\cdots
v_{\sigma(k+m-1)}]=0\label{c}
\end{eqnarray}
Where $$B(\sigma,l-1)=(-1)^{\sigma+i_{\sigma(1)}+\cdots +
i_{\sigma(l-1)}}e(\sigma)$$
\end{proposition}
\begin{proof}
Since for every ghost variable $\eta^j$, we can associate it with a
vector $v_j$, by the theorem 3,$S^2\eta^{j}=0$ iff
\begin{eqnarray}
\sum_{\sigma\in S_{k+m-1},1\leq l\leq k}\frac{1}{m!k!}
A(\sigma)C_{a,i_{\sigma(1)}\cdots i_{\sigma(l+m)}\cdots
i_{\sigma(k+m-1)}}^{j}C_{i_{\sigma(l)}\cdots
i_{\sigma(l+m-1}}^{a}v_{j} =0 \label{h}
\end{eqnarray}
Where $$A(\sigma)=(-1)^{\sigma+\sigma(l)+\cdots
+\sigma(l+m-1)+(a,\sigma(l),\cdots
\sigma(l-1))}e(\sigma)e(a,\sigma(l),\cdots \sigma(l-1))$$ That means
LHS of \ref{h} equal to

 $$\sum_{\sigma\in S_{k+m-1},1\leq l\leq k}\frac{1}{m!k!}
A(\sigma)C_{a,i_{\sigma(1)}\cdots i_{\sigma(l+m)}\cdots
i_{\sigma(k+m-1)}}^{j}l_{k}[v_a,v_{\sigma(1)}\cdots v_{\sigma(l-1)},
v_{\sigma(l+m)}\cdots v_{\sigma(k+m-1)}]$$

\begin{eqnarray}
=\sum_{\sigma\in S_{k+m-1},1\leq l\leq k}\frac{1}{m!k!} A(\sigma)
l_{k}[v_{\sigma(1)}\cdots v_{\sigma(l-1)},C_{i_{\sigma(l)}\cdots
i_{\sigma(l+m-1)}}^{a}v_a,
v_{\sigma(l+m)}\cdots v_{\sigma(k+m-1)}]\nonumber\\
=\frac{1}{m!k!} \sum_{\sigma\in S_{k+m-1},1\leq l\leq k}
 A(\sigma)l_{k}[v_{\sigma(1)}\cdots v_{\sigma(l-1)},l_{m}[v_{\sigma(l)}\cdots
v_{\sigma(l+m-1)}],v_{\sigma(l+m)}\cdots v_{\sigma(k+m-1)}]
\end{eqnarray}
\end{proof}
\bf{Remark.} Equations \ref{c} means skew multi-linear maps $l_n$
forms commutator sh Lie structure  .  In \cite{LS},the definition of
commutator sh Lie structure doesn't need $l_n$ to be skew symmetry
we would like to call this skew commutator sh Lie structure
commutator $L_\infty$ algebra denoted by $CL_\infty.$

\section{Linear representation of $CL_\infty$ algebra}

let $A$ be an algebra which contains $F$, we are going to extend
BRST operator from $R$ to $A\otimes R$ such that $S$ is still a
differential. To achieve this, let's consider a series of
skew-symmetry linear maps: $ \rho_{k}:\bigotimes^{n}V
\longrightarrow End(A)$ and for any $f\in A$ set
\begin{eqnarray}
Sf=\sum_{k=1}^{\infty}[\rho_{k}(v_{i_1},\cdots
v_{i_k})f]\eta^{i_1}\cdots \eta^{i_k} \label{d}
\end{eqnarray}
We need to determine the condition that $S^2=0$ when $S$ is extended
from $R$ to $A\otimes R$. Sometime, we denote
$\rho_{k}(v_{i_1},\cdots v_{i_k})f$ by $\rho_{i_1,\cdots i_k}$
\begin{theorem}
$S^2 =0$ iff
\begin{eqnarray}
=\sum_{k+l=n+1,\sigma \in
S_n}\frac{1}{l!}\sum_{m=1}^{k}(-1)^{\sigma}e(\sigma)\rho(v_{\sigma(1)},\cdots
v_{\sigma (m-1)},l_{l}[v_{\sigma(m)},\cdots v_{\sigma
(m+l-1)}],v_{\sigma(m+l)},\cdots v_{\sigma (n)}) f \label{d}
\end{eqnarray}
\end{theorem}
\begin{proof}
We still impose the condition that $S$ should be an odd derivative,
by the lemma 2, we have:
\begin{eqnarray}
& & S^2 f=S(\sum_{k=1}^{\infty}(\rho_{i_1,\cdots i_k})f
\eta^{i_1}\cdots
\eta^{i_k}) \nonumber\\
& & =\sum_{k=1}^{\infty}S(\rho_{i_1,\cdots i_k})f \eta^{i_1}\cdots
\eta^{i_k}+\sum_{k=1}^{\infty}(\rho_{i_1,\cdots i_k}f)
S(\eta^{i_1}\cdots \eta^{i_k})
\end{eqnarray}
For the convenience of simplification, we set the first part of
above sum to be $A$ and the second part to be$B$,then
\begin{eqnarray}
A=\sum_{k=1}^{\infty}[\sum_{l=1}^{\infty}\rho_{j_1,\cdots
j_l}(\rho_{i_1,\cdots i_k})f\eta^{j_1}\cdots \eta^{j_l}
]\eta^{i_1}\cdots \eta^{i_k}\nonumber\\
=\sum_{n=2}^{\infty}[\sum_{k+l=n}\rho_{i_1,\cdots
i_k}(\rho_{i_{k+1},\cdots i_{k+l}})f\eta^{i_1}\cdots \eta^{i_{k+l}}]
\end{eqnarray} and
\begin{eqnarray}
& & B=\sum_{k=1}^{\infty}(\rho_{i_1,\cdots i_k}f)(\sum_{m=1}^{k}
(-1)^{\eta^{i_1}+\cdots+\eta^{i_{m-1}}}
\eta^{i_1}\cdots S(\eta^{i_m})\cdots \eta^{i_k})\nonumber\\
& & =\sum_{k=1}^{\infty}(\rho_{i_1,\cdots
i_k}f)(\sum_{m=1}^{k}(-1)^{\eta^{i_1}+\cdots+\eta^{i_{m-1}}}\sum_{l=1}^{\infty}\frac{-1}{l!}C_{j_{1}\cdots
j_{l}}^{i_m}\eta^{i_1}\cdots
\eta^{i_{m-1}}[\eta^{j_1}\cdots\eta^{j_l}]\eta^{i_{m+1}}\cdots
\eta^{i_k})\nonumber\\
& &
=\sum_{n=1}^{\infty}[\sum_{k+l=n+1}\frac{1}{l!}\sum_{m=1}^{k}(\rho_{j_1,\cdots
i_{k-1}}f)(-1)^{\eta^{j_1}+\cdots+\eta^{j_{m-1}}}C_{j_{m}\cdots
j_{m+l-1}}^{i_m}\eta^{j_1}\cdots
\eta^{j_{m-1}}[\eta^{j_m}\cdots\eta^{j_{m+l-1}}]\eta^{j_{m+l}}\cdots
\eta^{i_{k+l-1}})]\nonumber\\
& &
=\sum_{n=1}^{\infty}[\sum_{k+l=n+1}\frac{1}{l!}\sum_{m=1}^{k}(\rho_{j_1,\cdots
i_{k-1}}f)(-1)^{\eta^{j_1}+\cdots+\eta^{j_{m-1}}}C_{j_{m}\cdots
j_{m+l-1}}^{i_m}\eta^{j_1}\cdots \eta^{j_{n}}]\nonumber\\
& &
=\sum_{n=1}^{\infty}[\sum_{k+l=n+1}\frac{1}{l!}\sum_{m=1}^{k}(\rho_{j_1,\cdots
a \cdots
i_{k-1}}f)(-1)^{\eta^{j_1}+\cdots+\eta^{j_{m-1}}}C_{j_{m}\cdots
j_{m+l-1}}^{a}\eta^{j_1}\cdots \eta^{j_{n}}]
\end{eqnarray}
Since ghost variables is graded commutative, we need to
anti-symmetrize terms $A$ and $B$ respectively.
\begin{eqnarray}
\hat{S}(A)=\sum_{n=2}^{\infty}[\sum_{k+l=n,\sigma \in
S_n}(-1)^{\sigma}e(\sigma)\rho_{i_(\sigma(1),\cdots i_{\sigma
(k)}}(\rho_{i_{\sigma(k+1)},\cdots
i_{\sigma(k+l)}})f]\eta^{i_1}\cdots \eta^{i_{k+l}}\nonumber\\
=\sum_{n=2}^{\infty}[\sum_{k+l=n,\sigma \in
S_n}(-1)^{\sigma}e(\sigma)\rho_{i_(\sigma(1),\cdots i_{\sigma
(k)}}(\rho_{i_{\sigma(k+1)},\cdots i_{\sigma(n)}})]f\eta^{i_1}\cdots
\eta^{i_{n}}
\end{eqnarray}
\begin{eqnarray}
\hat{S}(B)=\sum_{n=1}^{\infty}[\sum_{k+l=n+1,\sigma \in
S_n}\frac{1}{l!}\sum_{m=1}^{k}A(\sigma)(\rho_{j_{\sigma(1)},\cdots a
\cdots i_{\sigma(k-1)}}f)C_{j_{\sigma(m)}\cdots
j_{\sigma(m+l-1)}}^{a}\eta^{j_1}\cdots \eta^{j_{n}}]
\end{eqnarray}
Where
$$A(\sigma)=(-1)^{\sigma}e(\sigma)(-1)^{\eta^{j_{\sigma(1)}}+\cdots+\eta^{j_{\sigma(m-1)}}}$$
Then we get

$$\sum_{k+l=n,\sigma \in
S_n}(-1)^{\sigma}e(\sigma)\rho_{i_(\sigma(1),\cdots i_{\sigma
(k)}}(\rho_{i_{\sigma(k+1)},\cdots i_{\sigma(n)}})f $$
\begin{eqnarray}
=\sum_{k+l=n+1,\sigma \in
S_n}\frac{1}{l!}\sum_{m=1}^{k}(-1)^{\sigma}e(\sigma)(\rho_{j_{\sigma(1)},\cdots
a \cdots
i_{\sigma(k-1)}}f)(-1)^{\eta^{j_{\sigma(1)}}+\cdots+\eta^{j_{\sigma(m-1)}}}C_{j_{\sigma(m)}\cdots
j_{\sigma(m+l-1)}}^{a}
\end{eqnarray}
Without the loss of generality, we omit the subscript $j$ and use
the notation $\rho_{k}(v_1,\cdots v_k)f=\rho_{1,\cdots k}$, we get

$$\sum_{k+l=n,\sigma \in
S_n}(-1)^{\sigma}e(\sigma)\rho(v_{\sigma(1)},\cdots v_{\sigma
(k)})(\rho(v_{\sigma(k+1)},\cdots v_{\sigma(n)})f$$
\begin{eqnarray}
=\sum_{k+l=n+1,\sigma \in
S_n}\frac{1}{l!}\sum_{m=1}^{k}(-1)^{\sigma}e(\sigma)\rho(v_{\sigma(1)},\cdots
v_{\sigma (m-1)},l_{l}[v_{\sigma(m)},\cdots v_{\sigma
(m+l-1)}],v_{\sigma(m+l)},\cdots v_{\sigma (n)}) f \label{d}
\end{eqnarray}
Then follows the theorem.
\end{proof}
\begin{definition}
If $l_n: \bigotimes^{n} V\longrightarrow V$ is $CL_\infty$
algebra,and a collection of maps:$\rho_{k}:\bigotimes^{n}
V\longrightarrow End(A)$ satisfy equations \ref{d},we call
$\rho_{k}$ is a linear representation of $CL_\infty$ algebra on $A$.
\end{definition}
In the BRST theory or BV formulism, structure coefficient $C_{j_1
\cdots j_k}^{i}$ is not a constant anymore, they can take values in
some function ring. Then the pair$\rho_{k},l_n$ is not commutator sh
Lie structure and representation, but a kind of algebraoid, that
means the coefficient $C_{j_1 \cdots j_k}^{i}\in A$. An interesting
problem is that when $S^2\eta^{i}=0$ still holds if $C_{j_1 \cdots
j_k}^{i}\in A$. First of all, we need following definition:
\begin{definition}
We call $\rho_{k},l_n$ forms a $CL_{\infty}$ algebraoid if it
satisfies following equations:
$$\sum_{k+l=n,\sigma \in
S_n}\frac{1}{k!}[\chi(\sigma)(\rho_{i_{\sigma(1)},\cdots
i_{\sigma(m)}}C_{i_{\sigma(m+1)}\cdots i_{\sigma(k+m)}}^{j})$$
\begin{eqnarray}
+\sum_{l=1}^{k}(-1)^{i_{\sigma(1)+\cdots+i_{\sigma(l-1)}}}\hat{\chi(\sigma)}\frac{1}{(m+1)!}C_{i_{1}\cdots
i_{l-1},a\cdots i_{k+m-1}}^{j}C_{i_{l}\cdots i_{\sigma(l+m)}}^{a}]=0
\end{eqnarray}
\end{definition}
We have
\begin{theorem}
If $\rho_{k},l_n$ is a $CL_{\infty}$ algebraoid, $S^2=0$
\end{theorem}
\begin{proof}
Since \begin{eqnarray} & &
S^2\eta^{i}=S(-\sum_{k=1}^{\infty}\frac{1}{k!}C_{i_{1}\cdots
i_{n}}^{j}\eta^{i_1}\cdots\eta^{i_k})\nonumber\\
& & =-\sum_{k=1}^{\infty}\frac{1}{k!}[S(C_{i_{1}\cdots
i_{n}}^{j})\eta^{i_1}\cdots\eta^{i_k}+C_{i_{1}\cdots
i_{n}}^{j}(\eta^{i_1}\cdots\eta^{i_k})]
\end{eqnarray}
By using the calculation of the proof in theorem 3, we have

$$S^2\eta^{i}=-\sum_{k=1}^{\infty}\frac{1}{k!}[\sum_{l=1}^{\infty}(\rho_{j_1,\cdots
j_l}C_{i_{1}\cdots i_{n}}^{j})\eta^{j_1}\cdots \eta^{j_l}
\eta^{i_1}\cdots
\eta^{i_k}$$
\begin{eqnarray}
+\sum_{m=1}^{\infty}\frac{1}{m!}\sum_{l=1}^{k}(-1)^{i_1+\cdots+i_{l-1}}
C_{i_{1}\cdots i_{l-1},a\cdots i_{k+m-1}}^{j}C_{i_{l}\cdots
i_{l+m-1}}^{a}\eta^{i_1}\cdots\eta^{i_{k+m-1}}]\nonumber\\
\end{eqnarray}
$$=-\sum_{n=2}^{\infty}\sum_{k+l=n}\frac{1}{k!}(\rho_{i_1,\cdots
i_l}C_{i_{l+1}\cdots i_{l+k}}^{j})\eta^{i_1}\cdots
\eta^{i_{k+l}}$$
\begin{eqnarray}
-\sum_{n=2}^{\infty}\sum_{k+l=n+1}\sum_{l=1}^{k}(-1)^{i_1+\cdots+i_{l-1}}
C_{i_{1}\cdots i_{l-1},a\cdots i_{k+m-1}}^{j}C_{i_{l}\cdots
i_{l+m-1}}^{a}\eta^{i_1}\cdots\eta^{i_{k+m-1}}\nonumber\\
\end{eqnarray}
$$=-\sum_{n=2}^{\infty}\sum_{k+l=n}\frac{1}{k!}[\rho_{i_1,\cdots
i_m}C_{i_{m+1}\cdots i_{k+m}}^{j}$$
$$+\sum_{l=1}^{k}(-1)^{i_1+\cdots+i_{l-1}}\frac{1}{(m+1)!}C_{i_{1}\cdots
i_{l-1},a\cdots i_{k+m-1}}^{j}C_{i_{l}\cdots
i_{l+m}}^{a}]\eta^{i_1}\cdots\eta^{i_{k+m}}$$

$$=-\frac{1}{n!}\sum_{n=2}^{\infty}\sum_{\sigma \in
S_{n},k+l=n}\frac{1}{k!}[\chi(\sigma)(\rho_{i_{\sigma(1)},\cdots
i_{\sigma(m)}}C_{i_{\sigma(m+1)}\cdots i_{\sigma(k+m)}}^{j})$$
$$+\sum_{l=1}^{k}(-1)^{i_{\sigma(1)}+\cdots+i_{\sigma(l-1)}}\hat{\chi(\sigma)}\frac{1}{(m+1)!}C_{i_{1}\cdots
i_{l-1},a\cdots i_{k+m-1}}^{j}C_{i_{l}\cdots
i_{\sigma(l+m)}}^{a}]\eta^{i_{\sigma(1)}}\cdots\eta^{i_{\sigma(n)}}$$
Where $\chi(\sigma),\hat{\chi(\sigma)}$ are corresponding signs of
permutations. Therefore $S^2=0$ iff

$$\sum_{k+l=n,\sigma \in
S_n}\frac{1}{k!}[\chi(\sigma)(\rho_{i_{\sigma(1)},\cdots
i_{\sigma(m)}}C_{i_{\sigma(m+1)}\cdots
i_{\sigma(k+m)}}^{j})$$
\begin{eqnarray}
+\sum_{l=1}^{k}(-1)^{i_{\sigma(1)+\cdots+i_{\sigma(l-1)}}}\hat{\chi(\sigma)}\frac{1}{(m+1)!}C_{i_{1}\cdots
i_{l-1},a\cdots i_{k+m-1}}^{j}C_{i_{l}\cdots i_{\sigma(l+m)}}^{a}]=0
\end{eqnarray}
\end{proof}
\section{$CL_\infty$ algebra differential and $GA_{\infty}$ algebra differential}
From now on, we will define  $CL_\infty$  algebra differential for
any given pair $\rho_{n},l_m$ where $\rho_n$ is linear
representation of $CL_\infty$ algebra $l_m$. Let $V$ be a
$\rho_{n},l_m$ module,i.e$$\rho_n: \underbrace{L\times L\cdots
\times L\longrightarrow EndV}$$ Where $L$ is the underlying space of
BRST algebra $L_n$. Let $C^{n}(L,V)$ be a n-dimensional $V$-cochain
space, i.e. $$\omega_n: \underbrace{L\times L\cdots \times
L\longrightarrow V}$$is a skew-symmetry multi-linear map, set $
C^{*}(L,V)=\sum_{n=0}^{\infty}C^{n}(L,V)$. Then we can define a
coboundary operator on the complex $C^{*}(L,V)$:
\begin{definition}
For any $\omega_n\in C^{n}(L,V)$, we have $k$-component of BRST
operator : $$S_k: C^{n}(L,V)\longrightarrow C^{n+k-1}(L,V)$$ as
follows:

$$(S_k\omega_{n})(X_{j_1},\cdots
X_{j_{n+k-1}})$$

$$
=\frac{1}{(n+k-1)!}\sum_{\sigma\in S_{n+k-1}}
\chi(\sigma)\rho(X_{j_{\sigma(1)}}\cdots
X_{j_{\sigma(l-1)}})(\omega(X_{j_{\sigma(k)}}\cdots
X_{j_{\sigma(n+k-1)}})$$

$$-\frac{1}{k!}\frac{1}{(n+k-1)!}\sum_{l=1,\sigma\in
S_{n+k-1}}^{n}C(\sigma)\omega_{n}( X_{j_{\sigma(1)}}\cdots
X_{j_{\sigma(l-1)}},l_{k}[X_{j_{\sigma(l)}}\cdots
X_{j_{\sigma(k+l)}}]\cdots X_{j_{\sigma(n+k-1)}})$$ Where

$C(\sigma)=\chi(\sigma)(-1)^{j_{\sigma(1)}+\cdots+j_{\sigma(l-1)}}$
\end{definition}
Assume that $X_i,i\in I$ is a basis of $CL_\infty$ algebra $l_n$,
for every $X_i$, consider its dual element ghost $\eta^i$. Then
there is a natural correspondence between the complex $C^{*}(L,V)$
and the ghost complex $C[\eta^i]\otimes V$ via
$$\tilde{\omega_{n}}=\omega_{n}(X_{i_1},\cdots
X_{i_{n}})\eta^{i_1}\cdots \eta^{i_n}$$ If $\rho_{n},l_m$ is a
$CL_\infty$ representation pair , we can define an odd derivative
$\tilde{S}$ on $C[\eta^i]\otimes V$  as follows:
\begin{eqnarray}
\tilde{S_k}\eta^i=-\sum_{k=1}^{\infty}\frac{1}{k!}C_{j_{1}\cdots
j_{n}}^{i}\eta^{j_1}\cdots\eta^{j_k}
\end{eqnarray}
and
\begin{eqnarray}
\tilde{S_k}f=\rho(X_{j_1)}\cdots
X_{j_{k-1}})\eta^{j_1}\cdots\eta^{j_{k-1}}
\end{eqnarray}
We have
\begin{theorem}
\begin{eqnarray}
\tilde{S_k}\tilde{\omega_n}=\widetilde{S_{k}\omega_n}\label{e}
\end{eqnarray}
\end{theorem}
\begin{proof}
We will prove this theorem by solving $S_k$ in the equation\ref{e}
uniquely. Since
\begin{eqnarray}
& &\tilde{S_k}\tilde{\omega_n}=\tilde{S_k}(\omega_{j_1\cdots
j_n}\eta^{j_1}\cdots\eta^{j_{n}})\nonumber\\
& &=\tilde{S_k}(\omega_{j_1\cdots j_n})\eta^{j_1}\cdots\eta^{j_{n}}+
\omega_{j_1\cdots
j_n}\tilde{S_k}(\eta^{j_1}\cdots\eta^{j_{n}})\nonumber\\
& &=(\rho_{m_1\cdots m_{k-1}}\omega_{j_1\cdots
j_n})\eta^{m_1}\cdots\eta^{m_{k-1}}\eta^{j_1}\cdots\eta^{j_{n}}+\omega_{j_1\cdots
j_n}[\sum_{l=1}^{k} (-1)^{j_1+\cdots+j_{l-1}} \eta^{j_1}\cdots
\tilde{S_{k}(\eta^{i_l})}\cdots \eta^{j_n}]\nonumber
\end{eqnarray}
Set the first part to be $A$ ,the second part to be $B$,then
\begin{eqnarray}
& &A=(\rho_{j_1\cdots j_{k-1}}\omega_{j_k\cdots j_{n+k-1}})
\eta^{j_1}\cdots\eta^{j_{n+k-1}}\nonumber\\
& &=\frac{1}{(n+k-1)!}\sum_{\sigma\in S_{n+k-1}}
\chi(\sigma)(\rho_{j_{\sigma(1)}\cdots
j_{\sigma(k-1)}}\omega_{j_{\sigma(k)}\cdots j_{\sigma(n+k-1)}}
\eta^{j_1}\cdots\eta^{j_{n+k-1}}\nonumber\\
& &=\frac{1}{(n+k-1)!}\sum_{\sigma\in S_{n+k-1}}
\chi(\sigma)\rho(X_{j_{\sigma(1)}}\cdots
X_{j_{\sigma(l-1)}})(\omega(X_{j_{\sigma(k)}}\cdots
X_{j_{\sigma(n+k-1)}})\nonumber
\end{eqnarray}
\begin{eqnarray}
& &B=\omega_{j_1\cdots j_n}[\sum_{l=1}^{k} (-1)^{j_1+\cdots+j_{l-1}}
\eta^{j_1}\cdots [-\frac{1}{k!}C_{i_{1}\cdots
i_{n}}^{j_l}\eta^{i_1}\cdots\eta^{i_k}]\cdots \eta^{j_n}]\nonumber\\
& &=\omega_{j_1\cdots j_n}[\sum_{l=1}^{k} (-1)^{j_1+\cdots+j_{l-1}}
 (-\frac{1}{k!})C_{i_{1}\cdots
i_{n}}^{j_l}\eta^{j_1}\cdots [\eta^{i_1}\cdots\eta^{i_k}]\cdots \eta^{j_n}]\nonumber\\
& &=\sum_{l=1}^{k} (-1)^{j_1+\cdots+j_{l-1}}
 (-\frac{1}{k!})\omega_{j_1\cdots j_n}C_{i_{1}\cdots
i_{n}}^{j_l}\eta^{j_1}\cdots [\eta^{i_1}\cdots\eta^{i_k}]\cdots
\eta^{j_n}\nonumber\\
& &=\sum_{l=1}^{k} (-1)^{j_1+\cdots+j_{l-1}}
 (-\frac{1}{k!})\omega_{j_1\cdots j_{k-1},a,j_{l+k}\cdots j_{n+k-1}}C_{j_{l}\cdots
j_{l+k-1}}^{a}\eta^{j_1}\cdots \eta^{j_{n+k-1}} \nonumber\\
& &=-\frac{1}{k!}\sum_{l=1}^{k} (-1)^{j_1+\cdots+j_{l-1}}
  \omega_{n}(X_{j_1}\cdots X_{j_{l-1}},X_a,X_{j_{l+k}}\cdots X_{j_{n+k-1}})C_{j_{l}\cdots
j_{l+k-1}}^{a}\eta^{j_1}\cdots \eta^{j_{n+k-1}} \nonumber\\
& &=-\frac{1}{k!}\sum_{l=1}^{k} (-1)^{j_1+\cdots+j_{l-1}}
  \omega_{n}(X_{j_1}\cdots X_{j_{l-1}},l_{k}(X_{j_l}\cdots X_{j_{l+k-1}}),X_{j_{l+k}}\cdots X_{j_{n+k-1}})\eta^{j_1}\cdots \eta^{j_{n+k-1}} \nonumber\\
& &=\sum_{l=1,\sigma\in S_{n+k-1}}^{n}D(\sigma)\omega_{n}(
X_{j_{\sigma(1)}}\cdots
X_{j_{\sigma(l-1)}},l_{k}[X_{j_{\sigma(l)}}\cdots
X_{j_{\sigma(k+l)}}]\cdots X_{j_{\sigma(n+k-1)}} )\eta^{j_1}\cdots
\eta^{j_{n+k-1}}\nonumber
\end{eqnarray}
Where
$D(\sigma)=-\frac{1}{k!(n+k-1)!}(-1)^{j_{\sigma(1)}+\cdots+j_{\sigma(l-1)}}e(\sigma)$
\end{proof}
\begin{proposition}
The formal series $S=\sum_{k=1}^{\infty}S_k$ is a nilpotent operator
on the complex $C^{*}(L,V)$
\end{proposition}
We call above $S$ $CL_\infty$  algebra differential induced by
$CL_\infty$  algebra representation pair $(\rho_n, l_m)$.

So far we have studied $CL_\infty$ algebra differential,furthermore
it can be generalized to broader category which we will call
$GA_{\infty}$ algebra differential.
\begin{definition}
A series of multi-linear maps $$m_k: \bigotimes^{k}L\longrightarrow
L$$,if it satisfies following equations:
\begin{eqnarray}
\sum_{\Omega}(-1)^{\sigma(1)+\cdots+\sigma(l-1)}\chi(\sigma)m_{i}[
X_{\sigma(1)}\cdots X_{\sigma(l-1)},m_{j}[X_{\sigma(l)}\cdots
X_{\sigma(l+m-1)}]\cdots X_{\sigma(i+j-1)}]=0
\end{eqnarray}
Where $\Omega=(i+j=n+1,\sigma\in S_{n-1},1\leq l\leq i)$ We call
$m_k$ is a $GA_{\infty}$ algebra.
\end{definition}
\bigskip
{\bf Remark.} $GA_{\infty}$ essentially is commutator sh Lie
structure which is claimed as a special case of sh Lie algebra
\cite{LS},\cite{LM}
\bigskip

Obviously,$A_{\infty}$ algebra is a special case of generalized
$GA_{\infty}$ algebra.
\begin{definition}
Given a $GA_{\infty}$ algebra $m_k$ and series of maps $$\rho_n:
\bigotimes^{n}L\longrightarrow EndV$$ if it satisfies equations
\ref{d} we call $\rho_n$ is a $GA_{\infty}$ algebra  representation
with representation space $V$.
\end{definition}

Similarly, we can define  $GA_{\infty}$ differential and
$GA_{\infty}$ cohomology as follows:
\begin{definition}

For any $\omega_n\in C^{n}(L,V)$, we have $k$-component of
$GA_{\infty}$ algebra differential: $$S_k: C^{n}(L,V)\longrightarrow
C^{n+k-1}(L,V)$$ as follows:

$$(S_k\omega_{n})(X_{j_1},\cdots
X_{j_{n+k-1}})=\frac{1}{(n+k-1)!}\sum_{\sigma\in S_{n+k-1}}
\chi(\sigma)\rho(X_{j_{\sigma(1)}}\cdots
X_{j_{\sigma(l-1)}})(\omega(X_{j_{\sigma(k)}}\cdots
X_{j_{\sigma(n+k-1)}})$$
\begin{eqnarray}
+\sum_{\Omega}D(\sigma)\omega_{n}( X_{j_{\sigma(1)}}\cdots
X_{j_{\sigma(l-1)}},m_{k}[X_{j_{\sigma(l)}}\cdots
X_{j_{\sigma(k+l)}}]\cdots X_{j_{\sigma(n+k-1)}} )
\end{eqnarray}
Where $\Omega$ and $D(\sigma)$ are defined as before
\end{definition}

 By analogy of the proof of the theorem ,  $S=\sum_{k=1}^{\infty}S_k$ is also a nilpotent
 operator.
 If we loosen the condition that $\omega_n$ should be skew-symmetry
 and let $GA_{\infty}$ degenerate to $A_{\infty}$, we have
\begin{proposition}
If $m_k$ is a  $A_{\infty}$, then formal series
$S=\sum_{k=1}^{\infty}S_k$ defines a differential on
$C^{*}(L,V)=\bigoplus_{n=1}^{\infty}Hom(L^{\otimes n},V)$ Where

$$(S_k\omega_{n})(X_{j_1},\cdots X_{j_{n+k-1}})=\sum_{l=1}^{n}
\rho(X_{1}\cdots X_{l-1)})(\omega(X_{l}\cdots X_{n+k-1})$$
\begin{eqnarray}
-\sum_{l=1 }^{n}(-1)^{X_1+\cdots+ X_{l-1}}\omega_{n}( X_1\cdots
X_{l-1},m_{k}[X_{l}\cdots X_{k+l}]\cdots X_{n+k-1})
\end{eqnarray}
and $\rho_n$ is a $GA_{\infty}$ representation.
\end{proposition}
When $m_n$ is just an associate algebra $A$, there is a natural
$GA_{\infty}$ representation $\rho:A\longrightarrow A$with
representation space $A$ as follows: For any fixed $a\in A$ and any
$b\in A$ define $$\rho(a)b=m(a,b)$$, then the corresponding
differential is nothing but Hochchild differential.

\end{document}